\documentclass[12pt]{amsart}
\usepackage{amsfonts, amssymb, amscd, graphicx, dcpic, pictex, latexsym, amsmath, amsthm, mathrsfs}

\usepackage[mathscr]{euscript}
\usepackage[latin5]{inputenc}

\newcommand{\biset}{\Gamma}

\newcommand{\ra}{\rightarrow}

\def\out{{\rm Out}}
\def\ind{{\rm Ind}}
\def\res{{\rm Res}}
\def\coind{{\rm Coind}}
\def\inf{{\rm Inf}}
\def\infl{{\rm Inf}}
\def\defl{{\rm Def}}
\def\codef{{\rm Codef}}
\def\tin{{\rm Tin}}
\def\des{{\rm Des}}
\def\iso{{\rm Iso}}

\def\hom{{\rm Hom}}

\def\mod{{\rm mod}}

\def\sq{{\mbox{\rm Sq}}}
\def\morp{{\mbox{\rm Mor}}}
\def\vect{{\mbox{\rm vect}}}

\newtheorem{thm}{Theorem}[section]
\newtheorem{pro}[thm]{Proposition}
\newtheorem{cor}[thm]{Corollary}
\newtheorem{lem}[thm]{Lemma}

\newtheorem{prob}[thm]{Problem}
\newtheorem{rem}[thm]{Remark}

\title{Inducing native Mackey functors to biset functors}

\author{Olcay Coşkun}
\address{Boğaziçi University, Department of Mathematics 80815 Bebek, İstanbul, Turkey \\ olcay.coskun@boun.edu.tr }
\begin{document}
\maketitle
\begin{abstract}
In this paper, we describe the induction functor from the category of native Mackey functors to the category of biset functors for
a finite group $G$ over an algebraically closed field $k$ of characteristic zero. We prove two applications of this description. As
the first application, we exhibit that any projective biset functor over $k$ is induced from a (rational) virtual native Mackey functor. 
The second application is the explicit description of the projective indecomposable biset functors parameterized by simple groups.

\end{abstract}




\section{Introduction}

In this paper, we consider the induction functor from the category of native Mackey functors for a finite group $G$ over a field $k$ 
to the category of biset functors for $G$ over $k$. The theory of biset functors is introduced by Serge Bouc \cite{B96} as a unified 
theory of the classical operations in finite group representation theory. A basic result in this theory is the classification of the simple 
objects, which is done by Bouc. According to the classification, the simple functors are parameterized by the pairs $(H,V)$ where 
$H$ runs over all finite groups and $V$ runs over all simple $k\out(H)$-modules, both up to isomorphism. Here $\out(H)$ is the 
group of all outer automorphisms of $H$. On the other hand, the theory of global Mackey functors is an older theory which only deals 
with induction, restriction and transport of structure maps. We call a global Mackey functor defined only on the group $G$ a 
\mbox{\it{native Mackey functor}}, see Section \ref{sec:native} for details. 

In the context of biset functors, native Mackey functors can be identified as 
functors over bi-free bisets and hence there is a forgetful functor from the category of biset functors to that of native Mackey
functors. The left adjoint of this forgetful functor is the functor that we are interested in this paper.

More precisely, we let $G$ be a finite group and let $k$ be an algebraically closed field of characteristic zero and consider 
all biset functors defined only on the set of subquotients of $G$ together with the native Mackey functors for $G$. With this 
restriction on 
the characteristic of $k$, it is well-known that the category of global Mackey functors is semisimple, see \cite{Ba} or \cite{B96}. 
Similarly, in this case, the category of native Mackey functors is also semisimple. Moreover, as we show below, the induction functor 
is also easier to describe. 

As an application of this explicit description of the induction  functor, we are able to prove an induction theorem for projective 
biset functors. According to Theorem \ref{thm:induction}, any projective biset functor is a rational combination of induced 
native Mackey functor. 

We also consider some special cases of this induction functor. We explicitly describe the induction of a simple native Mackey
functor parameterized by a cyclic group. Also when $P$ is a $p$-group for some prime $p$ of order at least $p^3$ and $1$ is the 
trivial module, we show that the induction of the simple native Mackey functor parameterized by the pair $(P,1)$ contains at least
two summands, one of them being the projective cover of the torsion-free Dade group functor. The remaining case of the
elementary abelian $p$-group of rank $2$ is also considered. 

The last case we consider is the case where we induce a simple functor parameterized by a simple non-cyclic group. In this case,
 we 
show that the induced functor is indecomposable, and hence obtain a, fairly explicit, description of this indecomposable projective
functor. Finally, remark that, we have no other example of a non-simple group with the property that the corresponding induced 
simple native Mackey functor is indecomposable and leave the question whether the induced functor being indecomposable 
characterizes the simplicity of the group open.
\section*{Acknowledgement}
I would like to thank to the referee for his/her comments ard corrections.

\section{Preliminaries on biset functors}
In this section, we recall basics of biset functors and introduce our notation. We refer to \cite{B96} and \cite{C3} for details.
Let $G$ be a finite group and let $k$ be an algebraically closed field of characteristic zero. We denote by Sq$(G)$
the set of all subquotients of $G$, that is,
\[
\sq(G) = \{ H/N: N\unlhd H\le G \}.
\]
Following Bouc \cite{B96}, we define the \textit{biset category} $\mathcal C_G: = \mathcal C_{k,G}$ for $G$ over $k$ as the 
category whose set of objects are the groups in $\sq(G)$. Given $H, K\in \sq(G)$, we put
\[
\morp_{\mathcal C_G}(H,K) = k\otimes B(K\times H)
\]
where $B(K\times H)$ is the Grothendieck group of $(K,H)$-bisets. The composition of morphisms is the linear extension of 
the well-known amalgamated product of bisets.

Now a \textit{biset functor} for $G$ over $k$ is a $k$-linear functor from the biset category $\mathcal C_G$ for $G$ to the category 
${}_k\vect$ of
finite dimensional $k$-vector spaces. (Note that one can define biset functors for $G$ over any commutative ring with unity, but
in this paper, we concentrate only on this special case.) Further we denote by $\mathcal F_G: = \mathcal F_{k,G}$ the 
category of biset functors for $G$ over $k$, with the morphisms given by the natural transformations of biset functors. The 
category $\mathcal F_G$ is abelian. The parametrization of its simple objects follows from Bouc's parametrization of biset 
functors \cite{B96}. According to \cite[Theorem 4.3.10]{B96}, there is a bijective correspondence between the set of isomorphism
classes of simple biset functors for $G$ over $k$ and the set of pairs $(H,V)$ where $H$ is a subquotient of $G$ and $V$ is 
a simple $k\out(H)$-module, both taken up to isomorphism. We denote a representative of the isomorphism class of the simple 
biset functor corresponding to the pair $(H,V)$ by $S_{H,V}^G$. 

\begin{rem} The category $\mathcal C_G$ is a full subcategory of the biset category $\mathcal C$, defined in \cite{B96}.
By the general theory of induction-restriction, it is straightforward to conclude that the simple biset functors $S_{H,V}^G$ for $G$
are restrictions to $\mathcal C_G$ of the simple biset functors $S_{H,V}$ defined on the category $\mathcal C$. 
\end{rem}

Another way to define a biset functor for $G$ is to consider the category algebra of the category $\mathcal C_G$. 
Precisely, we let $\Gamma(G):= \Gamma_k(G)$ be the algebra generated by all morphisms in $\mathcal C_G$, with the 
multiplication induced by the composition of morphisms. Following Barker \cite{Ba}, we call $\Gamma(G)$ the \textit{alchemic
algebra} for $G$ over $k$. By Bouc's decomposition formula, \cite[Lemma 2.3.26]{B96}, the alchemic algebra is generated by
the set of all bisets given by the following list.
\begin{enumerate}
\item $\tin_{K/N}^H:= \ind_K^H\inf_{K/N}^K$ for all $N\unlhd K\le H\in \sq(G)$,
\item $\des_{K/N}^H:= \defl_{K/N}^K\res_K^H$ for all $N\unlhd K\le H\in \sq(G)$,
\item $\iso_{H,H'}^\lambda$ for all isomorphisms $\lambda:H'\ra H$ with $H,H'\in \sq(G)$.
\end{enumerate}
These generators are subject to a list of conditions determined by the Mackey product formula \cite[Lemma 2.3.24]{B96}. With
this definition, the category of finitely generated $\Gamma(G)$ modules is equivalent to the category $\mathcal F_G$ of biset
functors for $G$. The proof of the equivalence is standard. We refer to \cite{C3} for further details on this approach.
We only recover a special subalgebra structure of the alchemic algebra which will be useful later.

Let $\Delta:= \Delta(G):= \Delta_k(G)$ (resp. $\nabla:= \nabla(G):=\nabla_k(G)$) be the subalgebra of $\Gamma_k(G)$ generated 
by the bisets $\tin_K^H$ and $\iso_{H,H'}^\lambda$ (resp. $\des_K^H$ and $\iso_{H,H'}^\lambda$). We also denote by $\Omega:= 
\Omega(G):=\Omega_k(G)$ the subalgebra generated by the bisets $\iso_{H,H'}^\lambda$. Then we have the following triangle 
summarizing this subalgebra structure of the alchemic algebra.

\[
\begindc{\commdiag}[30]
\obj(70,10)[V]{$\Gamma$}
\obj(60,0)[B]{$\Delta$}
\obj(80,0)[C]{$\nabla$}
\obj(90,-10)[E5]{$\Omega$}
\obj(70,-10)[E4]{$\Omega$}
\obj(50,-10)[E3]{$\Omega$}
\mor{C}{V}{}[\atleft,\solidarrow]
\mor{B}{V}{}[\atleft,\solidarrow]

\mor{E5}{C}{}[\atleft,\solidarrow]
\mor{C}{E4}{}[\atleft,\solidarrow]
\mor{B}{E4}{}[\atleft,\solidarrow]
\mor{E3}{B}{}[\atleft,\solidarrow]

\mor(70,-10){E3}{}[\atleft, \solidline]
\mor(70,-11)(90,-11){}[\atleft, \solidline]
\mor(70,-10)(90,-10){}[\atleft, \solidline]

\mor(70,-11)(50,-11){}[\atleft, \solidline]
\enddc
\]
In the above diagram, all arrows going upward are inclusions, and the down arrows in the last row are quotient maps
from the corresponding algebra to its quotient by the ideal generated by all non-iso generators.
These algebra morphisms induce several induction, coinduction and restriction functors between the corresponding categories 
of modules. We also have several natural isomorphisms between these functors and natural equivalences between module
categories. The key result to prove such equivalences is the following theorem. Note that this is an alchemic version of Theorem
3.2 in \cite{C1} where the author uses a similar triangle in the context of ordinary Mackey functors. 

\begin{thm}\label{thm:tensor-alchemic}
Assume the above notation. Then there is an isomorphism 
\[
{}_\Delta\Gamma_\nabla \cong \Delta\otimes_\Omega \nabla
\]
of $(\Delta,\nabla)$-bimodules, where we regard $\Gamma$ as a $(\Delta,\nabla)$-bimodule via left and right multiplication.
\end{thm}
The proof of this theorem follows the same lines of the proof of Theorem 3.2 in \cite{C1}. Now the following theorem, which is
a version of Theorem 3.4 in \cite{C1}, also holds for in the case of the alchemic algebra. The proof is again almost the same
as the proof in \cite{C1}.
\begin{thm}\label{thm:alchemic-equivs}
The following equivalences hold.
\begin{enumerate}
\item $\res^\Gamma_\Delta\ind_\nabla^\Gamma \cong \ind_\Omega^\Delta\res^\nabla_\Omega$.
\item $\res^\Gamma_\nabla\coind_\Delta^\Gamma \cong \coind_\Omega^\nabla\res^\Delta_\Omega$.
\end{enumerate}
\end{thm}

\section{Native Mackey functors}\label{sec:native}

Other than the biset functors, there are several other natural constructions of functors associated to
bisets, like global Mackey functors, inflation functors and restriction functors. In this paper, we are mainly interested in a local
version of global Mackey functors, called \textit{native Mackey functors}. By definition, a native Mackey functor is a kind of
global Mackey functor which is defined only on the subquotients of a fixed finite group $G$. In other words, it is a biset functor for
$G$ without inflation and deflation maps. 
In this section, we introduce the
formal definition and basic properties of native Mackey functors. We shall use the module theoretic approach, however it is 
possible to define them in a way similar to the first definition of a biset functor. 

We denote by $\mu:= \mu(G):=\mu_k(G)$ the subalgebra of $\Gamma(G)$ generated by the bi-free bisets. In other words, it is the 
subalgebra of $\Gamma(G)$ generated by the following list of generators.
\begin{enumerate}
\item $\tin_{K}^H= \ind_K^H$ for all $K\le H\in \sq(G)$,
\item $\des_{K}^H= \res_K^H$ for all $K\le H\in \sq(G)$,
\item $\iso_{H,H'}^\lambda$ for all isomorphisms $\lambda:H'\ra H$ with $H,H'\in \sq(G)$.
\end{enumerate}
A module over $\mu_k(G)$ is called a \textit{native Mackey functor for $G$ over $k$}. We denote the category of all native Mackey functors 
for $G$ by ${}_{\mu(G)}\mod$ and call the algebra $\mu(G)$ the \textit{native Mackey algebra}.

Notice that a native Mackey functor is a structure very similar to the global Mackey functors, considered by Webb in \cite{W}. 
Indeed, if $M$ is a global Mackey functor, then the restriction of this functor to the set of all subquotients of $G$ would give a 
native Mackey functor. On the other hand, an ordinary Mackey functor is not necessarily a native Mackey functor, since in the 
case of ordinary Mackey functors, one is only allowed to consider isomorphisms induced by the conjugation action of $G$ on the
subgroups, whereas in our case, all isomorphisms between subquotients of $G$ are allowed. 

Moreover, the native Mackey algebra is an idempotent truncation 
\[
\mu_k(G) = e_G \mu_k^{1,1} e_G
\]
of the global Mackey algebra $\mu^{1,1}_k$ of \cite{W}. Here the idempotent $e_G$ is given by 
\[
e_G = \sum_{H\in\tiny{\sq}(G)}\iso_{H,H}^1.
\]

Therefore we 
can specialize results on the global Mackey functors in \cite{B96} and \cite{W} to the case of the native Mackey functors, easily. 
We give a couple of these results below. (cf. \cite[Theorem 9.5, ]{W} and \cite[Theorem 4.3.10]{B96})
\begin{thm}
Let $G$ be a finite group and $k$ be an algebraically closed field of characteristic zero. Then
\begin{enumerate}
\item The native Mackey algebra $\mu_k(G)$ is semisimple. 
\item The simple $\mu_k(G)$-modules are parameterized by the set of pairs $(H,V)$ where $H$ runs over the isomorphism 
classes of subquotients of $G$ and $V$ runs over the isomorphism classes of simple $k\out(H)$-modules. The simple functor
corresponding to the pair $(H,V)$ is denoted by $S_{H,V}^\mu$. It is characterized by the property that $H$ is the group of minimal
order such that $S_{H,V}(H)\neq 0$ and $V = S_{H,V}(H)$.
\end{enumerate}
\end{thm}
The proof of this theorem follows from the general results on modules of the truncated algebras. A construction of the simple 
native Mackey functors follows from the construction of simple global Mackey functors. However, in the rest of the paper, another
construction, similar to the construction of the simple ordinary Mackey functors given in \cite{C1} will be more useful. For the rest
of the section, we adapt certain results from \cite{C1} to the native case, without formal proofs.

First, we need to introduce a subalgebra structure of $\mu(G)$. This structure is very similar to the structure introduced above for
the alchemic algebra and hence to the structure introduced in \cite{C1} for the ordinary Mackey algebra. Therefore the results
that are recalled above for the alchemic algebra will also hold for the native Mackey algebra, as explained below.

We denote by
$\tau:= \tau(G):=\tau_k(G)$ (resp. $\rho:= \rho(G):=\rho_k(G)$) the subalgebra of $\mu$ generated by the bisets $\ind_K^H$ and 
$\iso_{H,H'}^\lambda$ (resp. $\res_K^H$ and $\iso_{H,H'}^\lambda$). We still denote by $\Omega$ the
subalgebra generated by the bisets $\iso_{H,H'}^\lambda$. 

Together with the alchemic algebra, we can represent this subalgebra structure of the alchemic algebra by the following diagram.
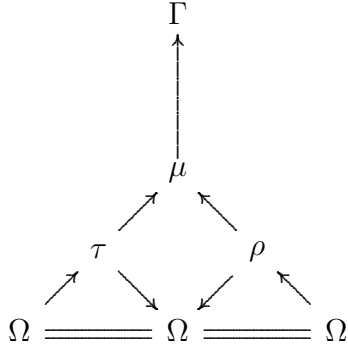
\[
\begindc{\commdiag}[30]
\obj(70,30)[X]{$\Gamma$}
\obj(70,10)[V]{$\mu$}
\obj(60,0)[B]{$\tau$}
\obj(80,0)[C]{$\rho$}
\obj(90,-10)[E5]{$\Omega$}
\obj(70,-10)[E4]{$\Omega$}
\obj(50,-10)[E3]{$\Omega$}

\mor{V}{X}{}[\atleft,\solidarrow]

\mor{C}{V}{}[\atleft,\solidarrow]
\mor{B}{V}{}[\atleft,\solidarrow]

\mor{E5}{C}{}[\atleft,\solidarrow]
\mor{C}{E4}{}[\atleft,\solidarrow]
\mor{B}{E4}{}[\atleft,\solidarrow]
\mor{E3}{B}{}[\atleft,\solidarrow]

\mor(70,-10){E3}{}[\atleft, \solidline]
\mor(70,-11)(90,-11){}[\atleft, \solidline]
\mor(70,-10)(90,-10){}[\atleft, \solidline]

\mor(70,-11)(50,-11){}[\atleft, \solidline]
\enddc
\]

In this diagram, as in the previous one, all arrows going upward are inclusions, and the down arrows in the last row are quotient 
maps from the corresponding algebra to its quotient by the ideal generated by all non-iso generators. Again, there are several
natural equivalences between functors induced by these maps.
 We refer to \cite{C1} and \cite{C3} for the full descriptions of such equivalences. We only note that the equivalences
 in \cite{C1} that are stated in the case of the ordinary Mackey algebra are also valid for the native Mackey algebra. In this paper,
we only need several of these equivalences. For the readers convenience, we shall state the native versions of the results 
that we use in this paper. Finally, note that there is an ambiguity in our notation since the ordinary Mackey algebra is also 
denoted by $\mu$ and the ordinary versions of the algebras $\tau$ and $\rho$ above are also denoted by the same symbol.
However, we will not use the ordinary versions of these algebras in this paper. So the ambiguity should not be a problem.

The first result determines the $(\tau,\rho)$-bimodule structure of the native Mackey algebra $\mu(G)$ and it is actually the key 
for the rest of the equivalences.
\begin{thm}\label{thm:tensor-bimodule}
Assume the above notation. Then there is an isomorphism 
\[
{}_\tau\mu_\rho \cong \tau\otimes_\Omega \rho
\]
of $(\tau,\rho)$-bimodules, where we regard $\mu$ as a $(\tau,\rho)$-bimodule via left and right multiplication.
\end{thm}
As we have remarked in the case of the alchemic algebra, the proof of this theorem follows the same steps of the proof of 
\cite[Theorem 3.2]{C1} with the only difference that this time we are working with isomorphisms instead of conjugation maps.
Using this result, one gets the following natural equivalences (cf. \cite[Theorem 3.4 and Proposition 3.6]{C1}). 
\begin{thm}\label{thm:native-equivs}
The following equivalences hold.
\begin{enumerate}
\item $\res^\mu_\tau\ind_\rho^\mu \cong \ind_\Omega^\tau\res^\rho_\Omega$.
\item $\res^\mu_\rho\coind_\tau^\mu \cong \coind_\Omega^\rho\res^\tau_\Omega$.
\item $\defl^\tau_\Omega\ind_\Omega^\tau\cong {\mbox{\rm Id}}_\Omega \cong \codef^\rho_\Omega\coind_\Omega^\rho.$
\item $\defl^\tau_\Omega \infl_\Omega^\tau \cong {\mbox{\rm Id}}_\Omega \cong \codef^\rho_\Omega\infl_\Omega^\rho.$
\end{enumerate}
\end{thm}
The proof of this theorem is again very similar to the proof of \cite[Theorem 3.4]{C1}. Indeed, the proof of the cited theorem 
only uses general results together with \cite[Theorem 3.2]{C1}, whose native version is given above. 

Now the construction of simple functors can be carried out exactly as the construction of simple Mackey functors explained in
Section 6 of \cite{C1}. Indeed, it is easy to see that any simple $\Omega_k(G)$-module is of the form $S_{H,V}^\Omega$
where $S_{H,V}^\Omega (H) = V$ and its evaluation is equal to zero at any other subquotient not isomorphic to $H$. Here
$H$ is a subquotient of $G$ and $V$ is a simple $k\out(H)$-module. Then
as in \cite[Proposition 6.1]{C1}, the simple $\tau$-modules and the simple $\rho$-modules are just the corresponding inflations of
the simple $\Omega$-modules, that is, any simple $\rho$-module is of the form 
\begin{equation}\label{eqn:simple0}
S_{H,V}^\rho := \infl_{\Omega}^\rho S_{H,V}^\Omega
\end{equation}
as $S_{H,V}^\Omega$ runs over all simple $\Omega$-modules. Similar remark holds for simple $\tau$-modules.

 Finally, we have isomorphisms
\begin{equation}\label{eqn:simples}
S_{H,V}^\mu\cong \ind_\rho^\mu\infl_\Omega^\rho S_{H,V}^\Omega\cong \coind_\tau^\mu\inf_\Omega^\tau S_{H,V}^\Omega.
\end{equation}
Here the simplicity of the induced and the coinduced functors follows easily since they are indecomposable and the native 
Mackey algebra is semisimple. They are also isomorphic by Schur's Lemma since there is a non-zero map between them induced by 
the identity homomorphism $V\to V$, see Corollary 5.10 in \cite{C1} for the derivation of this map. Finally, the first isomorphism is just 
the identification coming from the parametrization of the simple functors. 

Note also that, since the native Mackey algebra is semisimple, the first isomorphism above actually gives an equivalence of 
categories
\[
\ind_{\rho}^{\mu}\inf_{\Omega}^{\rho}: {}_\Omega\mod\ra {}_\mu\mod
\]
with the inverse equivalence 
\[
\defl_{\Omega}^{\tau}\res^{\mu}_{\tau}: {}_\mu\mod\ra {}_\Omega\mod.
\]
Indeed, the composition $\defl_{\Omega}^{\tau}\res^{\mu}_{\tau}\ind_{\rho}^{\mu}\inf_{\Omega}^{\rho}$ is naturally isomorphic
to the identity functor on $\Omega$-mod by Theorem \ref{thm:native-equivs}. To see that the other composition is naturally 
isomorphic to the identity functor on $\mu$-mod, one needs to show that there is an 
isomorphism
\begin{equation}\label{eqn:twin}
M\cong \ind_{\rho}^{\mu}\inf_{\Omega}^{\rho}\defl_{\Omega}^{\tau}\res^{\mu}_{\tau} M
\end{equation}
of native Mackey functors for any native Mackey functor $M$. Since $\mu$ is semisimple, it is sufficient to prove the 
isomorphism only for the simple native Mackey functors. So let $M= S_{H,V}^\mu$ for some $(H,V)$. Then by the definition of 
the deflation functor, for any subquotient $K$ of $G$, the module $\defl_{\Omega}^{\tau}\res^{\mu}_{\tau} M (K)$ is the quotient 
of $M(K)$  by the sum of images of all induction maps to $K$. But since $M$ is simple, this quotient will be non-zero only if 
$K\cong H$ in which case, it will be equal to $M(K)$. In particular, there is an isomorphism
\[
\defl_{\Omega}^{\tau}\res^{\mu}_{\tau} S_{H,V}^\mu \cong S_{H,V}^\Omega.
\]
Now the result follows form the above construction of the simple native Mackey functors.
\section{The induction functor $\ind_{\mu(G)}^{\Gamma(G)}$}\label{evaluationSection}
We are interested in the induction functor $\ind_\mu^\Gamma:= \ind_{\mu(G)}^{\Gamma(G)} = \Gamma(G)\otimes_{\mu(G)} -$.
Note that, in general, the right $\mu(G)$-module structure of $\Gamma(G)$ is very complicated and hence it is difficult to describe 
the induction functor. However, in our case, when the coefficients are from a field of characteristic zero, the problem can be 
simplified by using the equivalences of the previous section. Explicitly, we have the following result.

\begin{thm}\label{isomtheo} Let $M$ be a native Mackey functor for $G$ over $k$. Then there is an isomorphism of biset functors
\[
\ind_\mu^\biset\, M \cong \ind_\rho^\biset\, \inf^\rho_\Omega \, \defl^\tau_\Omega \, \res^\mu_\tau \, M.
\]
\end{thm}
\begin{proof} Note that since the algebra $\rho$ is a subalgebra of the native Mackey algebra $\mu$, we can rewrite the right hand side
of the above isomorphism as
\[
\ind_\rho^\biset\, \inf^\rho_\Omega \, \defl^\tau_\Omega \, \res^\mu_\tau \, M =
\ind_\mu^\biset\, \ind_\rho^\mu\, \inf^\rho_\Omega \, \defl^\tau_\Omega \, \res^\mu_\tau \, M.
\]
Therefore it suffices to show that there is an isomorphism of native Mackey functors
\[
M \cong \ind_\rho^\mu\, \inf^\rho_\Omega \, \defl^\tau_\Omega \, \res^\mu_\tau \, M
\]
which follows from the category equivalence explained at the end of the previous section.
\end{proof}

Now since the native Mackey algebra $\mu_k(G)$ is semisimple and the induction functor preserves projectivity, we obtain the
following corollary on projectivity of the induced functors.

\begin{cor}\label{cor1proj}
Let $M$ be a native Mackey functor for $G$ over $k$. Then the biset functor $\ind_\mu^\biset\, M$ is projective.
\end{cor}

\begin{rem} To use the above identification more efficiently, one can consider the induction functor 
$\ind_\rho^\Gamma$ as a composition $\ind_\nabla^\Gamma\ind_\rho^\nabla$ where $\nabla$ is the subalgebra of $\Gamma$
generated by all morphisms of the form $\iso_{K',K/N}^\lambda\des^H_{K/N}$ for all appropriate choices of the subquotients 
$H,K/N$
and $K'$ and isomorphisms $\lambda$. Now the description of $\ind_\nabla^\Gamma$ can be found in \cite{C3} and we include
a description of the other functor in the appendix.
\end{rem}
\section{Projective biset functors}
In this section, we prove that any projective biset functor for $G$ over $k$ is obtained by inducing a virtual native Mackey functor.
We first introduce our notation. Let $\mathcal R_{\mu,k}(G)$ denote the Grothendieck group of the category of native Mackey 
functors for $G$ over $k$ and let $\mathcal P_{\Gamma, k}(G)$ denote the Grothendieck group of the category of projective biset 
functors for 
$G$ over $k$. Note that both of these groups are free of the same rank with basis parameterized by the pairs $(H,V)$ described 
in the previous sections. In the first case, the basis consists of the set of isomorphism classes of simple native Mackey functors
$S_{H,V}^\mu$, whereas, in the second case, the basis consists of the set of isomorphism classes of indecomposable
projective biset functors $P_{H,V}^\Gamma$.

Now we have the following theorem.  
\begin{thm}\label{thm:induction} The induction functor
\[
\ind_{\mu(G)}^{\Gamma(G)}: {}_{\mu}\mod \rightarrow {}_{\Gamma}\mod
\]
induces an isomorphism
\[
\ind_{\mu(G)}^{\Gamma(G)}: \mathbb Q\mathcal R_{\mu, k}(G) \rightarrow \mathbb Q\mathcal P_{\Gamma, k}(G).
\]
of $\mathbb Q$-vector spaces.
\end{thm}
\begin{proof} For simplicity, we write $\Gamma$ and $\mu$, instead of $\Gamma(G)$ and $\mu(G)$.
Evidently it suffices to prove that any projective indecomposable biset functor is a rational combination of induced native Mackey 
functors. For this aim, let $H$ be a subquotient of $G$ and let $V$ be a simple $k\out(H)$-module. Then, by Corollary 
\ref{cor1proj}, the biset functor 
$\ind_\mu^\biset S_{H,V}^\mu$ is projective, so we have
\[
\ind_\mu^\biset\, S_{H,V}^\mu \cong \bigoplus_{(K,W)}n_{K,W}^{H,V}\, P_{K,W}^\biset
\]
where the sum is over the set of isomorphism classes of simple biset functors for $G$ over $k$ and $P_{K,W}^\biset$ denotes the projective cover of
the simple functor $S_{K,W}^\biset$. Since $k$ is algebraically closed, we further have that
\[
n_{K,W}^{H,V} = {\rm dim}_k\,\hom_\biset(\ind_\mu^\biset\, S_{H,V}^\mu, S_{K,W}^\biset).
\]
Now, by Equation \ref{eqn:simple0} and Equation \ref{eqn:simples}, there is an isomorphism of 
native Mackey functors
\begin{eqnarray}\label{eqn:simple-iso}
S_{H,V}^\mu \cong \ind_\rho^\mu\, S_{H,V}^\rho.
\end{eqnarray}
Here the simple $\rho$-module $S_{H,V}^\rho$ is the functor which takes the value zero for all subquotients of $G$ except for 
the isomorphism class of $H$ and $S_{H,V}^\rho(H) = V$. Substituting this into the above equality,
we get
\[
n_{K,W}^{H,V} = {\rm dim}_k\,\hom_\biset(\ind_\rho^\biset S_{H,V}^\rho, S_{K,W}^\biset).
\]
On the other hand, since the induction functor is the left adjoint of the restriction functor, we get
\[
n_{K,W}^{H,V} = {\rm dim}_k\,\hom_\rho(S_{H,V}^\rho, \res^\biset_\rho S_{K,W}^\biset).
\]
Note that, to finish the proof, it is sufficient to show that the matrix
\[
N_G = \big( n_{K,W}^{H,V} \big)_{((K,W),(H,V))}
\]
is invertible. We proceed to show that, with a correct ordering, the matrix $N_G$ is lower triangular with
diagonal entries equal to 1.

Indeed let $\mathcal H := \{ 1= H_0, H_1, \cdots, H_n=G \}$ be a set of representatives of the isomorphism classes of 
subquotients of $G$ ordered in a way that if $i< j$ then $|H_i|\le |H_j|$. Also let
\[
\mathcal S := \{ (H_i,V_{ij})| H_i\in \mathcal H, V_{ij}\in {\rm Irr}(k\out(H_i)) \}.
\]
Then we claim that
\[
n_{H_i,V_{ij}}^{H_k,V_{kl}} = 0
\]
if $i>k$. Indeed, since the $\rho$-module $S_{H_k,V_{kl}}^\rho$ is non-zero only on the isomorphism class of the group $H_k$,
any $\rho$-module homomorphism 
\[
\phi: S_{H_k, V_{kl}}^\rho \to \res^\Gamma_\rho S_{H_i,V_{ij}}^\Gamma
\]
is uniquely determined by its component
\[
\phi_{H_k}: S_{H_k, V_{kl}}^\rho(H_k) \to \res^\Gamma_\rho S_{H_i,V_{ij}}^\Gamma(H_k).
\]
Moreover the map $\phi_{H_k}$ should be a homomorphism of $k\out(H_k)$-modules by its definition. Thus there is an inclusion 
\[
\hom_\rho(S_{H_k, V_{kl}}^\rho, \res^\Gamma_\rho S_{H_i,V_{ij}}^\Gamma)\subseteq \hom_{k\out(H_k)}(V_{kl}, S_{H_i,V_{ij}}^\Gamma(H_k))
\]
of $k$-vector spaces. Therefore we get 
\begin{eqnarray}\label{eqn:ineqn}
n_{H_i,V_{ij}}^{H_k,V_{kl}} = {\rm dim}_k\,\hom_\rho(S_{H_k,V_{kl}}^\rho, \res^\biset_\rho S_{H_i,V_{ij}}^\biset)\le
{\rm dim}_k\,\hom_{k\out(H_k)}(V_{kl}, S_{H_i,V_{ij}}^\biset(H_k)). 
\end{eqnarray}
But $H_i$ is a minimal group for the simple functor $S_{H_i,V_{ij}}^\Gamma$. Thus we have $S_{H_i,V_{ij}}^\Gamma(H_k) = 0$ 
which implies that the multiplicity $n_{H_i,V_{ij}}^{H_k,V_{kl}}$ is also zero, as claimed. 

On the other hand, if $i=k$, then there are two cases. If $j\neq l$, then by the inequality (\ref{eqn:ineqn}), we get 
\begin{eqnarray*}
n_{H_i,V_{ij}}^{H_i,V_{il}} &=& {\rm dim}_k\,\hom_\rho(S_{H_i,V_{il}}^\rho, \res^\biset_\rho S_{H_i,V_{ij}}^\biset)\\
&\le& {\rm dim}_k\,\hom_{k\out(H_i)}(V_{il}, V_{ij}).
\end{eqnarray*}
But since $j\neq l$, the module $V_{ij}$ is not isomorphic to $V_{il}$, so the latter dimension is equal to zero by Schur's Lemma. Finally, 
if $j=l$, then we have
\begin{eqnarray*}
n_{H_i,V_{ij}}^{H_i,V_{ij}} &=& {\rm dim}_k\,\hom_\rho(S_{H_i,V_{ij}}^\rho, \res^\biset_\rho S_{H_i,V_{ij}}^\biset)\\
&\le& {\rm dim}_k\,\hom_{k\out(H_i)}(V_{ij}, V_{ij}).
\end{eqnarray*}
This time, by Schur's Lemma, the later dimension is equal to 1. Now since the identity morphism $\iota:V_{ij}\to V_{ij}$ induces
a non-zero homomorphism 
\[
\iota^*: S_{H_i,V_{ij}}^\rho \to \res^\biset_\rho S_{H_i,V_{ij}}^\biset
\]
of $\rho$-modules, the set $\hom_\rho(S_{H_i,V_{ij}}^\rho, \res^\biset_\rho S_{H_i,V_{ij}}^\biset)$ is non-zero, and hence has 
dimension 1. Therefore, we obtain that
\[
n_{H_i,V_{ij}}^{H_i,V_{il}} = \delta_{jl}
\]
and hence combining this result with the previous case, we conclude that the matrix $$N_G = \big(n_{H_i,V_{ij}}^{H_k,V_{kl}}
\big)_{(H_i,V_{ij}), (H_k,V_{kl})\in \mathcal S}$$ is lower triangular having ones in the diagonal, as required. 
\end{proof}
The following corollary is immediate from the proof of the above theorem.
\begin{cor} Let $H,K\in \sq(G)$ and $V$ be a simple $k\out(H)$-module and $W$ be a simple 
$k\out(K)$-module. Then the multiplicity of the projective indecomposable biset functor $P_{K,W}^\Gamma$ as a direct summand 
in the biset functor $\ind_\mu^\Gamma S_{H,V}^\mu$ is non-zero only if $K$ is a subquotient of $H$.
\end{cor}

In the next section, a better description of this coefficient will be obtained by considering the off-diagonal entries of the above matrix. However, when $H$ is equal to $K$ in the above corollary, we get the following precise result.
\begin{cor}\label{corollary:top} Let $H\in \sq(G)$ and $V, W$ be simple $k\out(H)$-modules. Then the multiplicity of the projective indecomposable functor $P_{H,W}^\Gamma$ as a direct summand in the biset functor $\ind_\mu^\Gamma S_{H,V}^\mu$ is 
non-zero if and only if $V\cong W$. When the coefficient is non-zero, it is equal to 1.
\end{cor}
Thus for any simple native Mackey functor $S_{H,V}^\mu$, we have
\[
\ind_\mu^\Gamma S_{H,V}^\mu \cong P_{H,V}^\Gamma \oplus P
\]
where $P$ is a (possibly zero) projective biset functor having no summands of the form $P_{H,W}^\Gamma$.
\section{Off-diagonal entries of $N_G$} 
For the rest of the paper, we study the summand $P$ defined above, in some special cases. The following 
characterization and estimation of the off-diagonal entries will be useful.

In the proof of Theorem \ref{thm:induction}, we have seen the following equality.
\[
n_{K,W}^{H,V} = {\rm dim}_k \hom_\rho(S_{H,V}^\rho, \res^\biset_\rho S_{K,W}^\biset).
\]
Now, as remark in the previous section, we have that $S_{H,V}^\rho = \infl_{\Omega}^\rho S_{H,V}^\Omega$. Therefore, we have
\[
n_{K,W}^{H,V} = {\rm dim}_k \hom_\rho(\infl_\Omega^\rho S_{H,V}^\Omega, \res^\biset_\rho S_{K,W}^\biset).
\]
If we denote the right adjoint of $\infl_\Omega^\rho$ by $\codef^\rho_\Omega$, the above equality becomes
\begin{eqnarray*}
n_{K,W}^{H,V} &=& {\rm dim}_k \hom_\Omega(S_{H,V}^\Omega, \codef^\rho_\Omega\res^\biset_\rho S_{K,W}^\biset)\\
&=& {\rm dim}_k \hom_{k\out(H)}(V, \codef^\rho_\Omega\res^\biset_\rho S_{K,W}^\biset (H)).
\end{eqnarray*}
Here we use the Morita equivalence, described in \cite[Section 4]{C3} of the algebras $\Omega(G)$ and $\prod_H k\out(H)$ 
where the product is over all subquotients of $G$, up to isomorphism.

Since the algebra $\Omega$ is the quotient of the algebra $\rho$ by the ideal generated by all proper restriction bisets, the 
evaluation of $\codef^\rho_\Omega D$ for a $\rho$-module $D$ at a subquotient $K$ of $G$ is given by
\[
\codef^\rho_\Omega D (K) = \bigcap_{L<K} \ker(\res^K_L: D(K) \ra D(L)).
\]
Now the following proposition is immediate since the group algebra $k\out{H}$ is semisimple.

\begin{pro}\label{codefres} The multiplicity $n_{K,W}^{H,V}$ of $P_{K,W}^\Gamma$ in the biset functor 
$\ind_\mu^\Gamma S_{H,V}^\mu$ is equal to the multiplicity of the $k\out(H)$-module $V$ in the $k\out(H)$-module 
$\codef^\rho_\Omega\res^\biset_\rho S_{K,W}^\biset (H)$.
\end{pro}
 
Note that, in general, the evaluations of simple biset functors are not easy to determine. Therefore, the above 
numbers are not easy to calculate. Next, we derive an upper bound for this number which will be useful to determine if the 
number is zero or not.

To start with, note that
\[
n_{K,W}^{H,V}= \dim_k\hom_\Gamma(\ind_\mu^\Gamma S_{H,V}^\mu, S_{K,W}^\Gamma)= {\rm dim}_k 
\hom_\Gamma(\ind_\rho^\Gamma S_{H,V}^\rho, S_{K,W}^\biset) 
\]
by the isomorphism (\ref{eqn:simple-iso}) of the previous section. On the other hand,  by \cite[Corollary 4.7]{C3}, the simple functor 
$S_{K,W}^\Gamma$ is the unique minimal subfunctor of the functor 
$\coind_\Delta^\Gamma S_{K,W}^\Delta$. Here $\Delta$ denotes the subalgebra of $\Gamma$ generated by all bisets of the 
form $\tin_{L/N}^H$ and $\iso_{H,H'}^\lambda$. Thus any morphism of $\Gamma$-modules $\ind_\rho^\Gamma S_{H,V}^\rho
\ra S_{K,W}^\Gamma$ can be regarded as a morphism $\ind_\rho^\Gamma S_{H,V}^\rho \ra \coind_{\Delta}^\Gamma 
S_{K,W}^\Delta$.Thus we get 
\[
n_{K,W}^{H,V} \le \dim_k\hom_\Gamma(\ind_\rho^\Gamma S_{H,V}^\rho, \coind_\Delta^\Gamma S_{K,W}^\Delta).
\]
The right hand side of the above inequality can be calculated as follows.
\begin{eqnarray*}
\hom_\Gamma(\ind_\rho^\Gamma S_{H,V}^\rho, \coind_\Delta^\Gamma S_{K,W}^\Delta) &\cong& 
\hom_\nabla(\ind_\rho^\nabla S_{H,V}^\rho, \res^\Gamma_\nabla\coind_\Delta^\Gamma S_{K,W}^\Delta) \\
&\cong& \hom_\nabla(\ind_\rho^\nabla S_{H,V}^\rho, \coind_\Omega^\nabla\res^\Delta_\Omega S_{K,W}^\Delta)\\
&\cong& \hom_\Omega(\res^\nabla_\Omega\ind_\rho^\Gamma S_{H,V}^\rho, \res^\Delta_\Omega S_{K,W}^\Delta).
\end{eqnarray*}

Here we use Theorem \ref{thm:alchemic-equivs} to obtain the second line, and then use the well-known adjointness properties of 
restriction and induction-coinduction functors to get the last line. 

Now since $S_{K,W}^\Delta$ is non-zero only on the isomorphism class of $K$, the last term in the above isomorphisms gives
the following isomorphism.
\begin{eqnarray*}
\hom_\Gamma(\ind_\rho^\Gamma S_{H,V}^\rho, \coind_\Delta^\Gamma S_{K,W}^\Delta) 
&\cong& \hom_{k\out(K)}(\ind_\rho^\nabla S_{H,V}^\rho (K), W).
\end{eqnarray*}

Thus the coefficient $n_{K,W}^{H,V}$ is non-zero only if $W$ appears in the $k\out(K)$-module 
$\ind_\rho^\nabla S_{H,V}^\rho (K)$ as a direct summand. To determine this number, we need a description of the induction functor 
$\ind_\rho^\nabla$ which is given in the appendix. Now using the description in Theorem \ref{pro:rhoTonabla}, we have
\begin{eqnarray*}
\hom_\Gamma(\ind_\rho^\Gamma S_{H,V}^\rho, \coind_\Delta^\Gamma S_{K,W}^\Delta)\cong&& \\
\bigoplus_{R\in{\tiny\sq}(G): R\cong H}\hom_{k\out(K)}(kX(K,R) \otimes_{k\out(R)} V(R), W)
\end{eqnarray*}
where $V(R) = S_{H,V}^\rho(R)$ and $X(K,R)$ is basically the set of isomorphisms from quotients of $R$ to $K$, see the 
appendix for the precise definition. As a result, we have the following proposition.
\begin{pro} Assume the above notation. Then the projective indecomposable biset functor 
$P_{K,W}^\Gamma$ occurs as a direct summand of $\ind_\mu^\Gamma S_{H,V}^\mu$ only if $W$ occurs as a direct summand 
of the $k\out(K)$-module $kX(K,R)\otimes_{k\out(R)}V(R)$ for some subquotient $R$ of $G$ isomorphic to $H$.
\end{pro}

In particular, we have the following corollary.
\begin{cor}\label{quotient} Assume the above notation. Then the projective indecomposable biset functor 
$P_{K,W}^\Gamma$ occurs as a direct summand of $\ind_\mu^\Gamma S_{H,V}^\mu$ only if there is a normal subgroup $N$
of $H$ such that $K\cong H/N$.
\end{cor}

\section{Induced functors for cyclic groups}
In this section, we describe the biset functor $\ind_\mu^\Gamma S_{C,V}^\mu$ where $C$ is a cyclic subquotient of $G$. First, by 
Corollary \ref{quotient}, any summand of this functor is determined by a quotient of $C$. Thus we can write
\begin{equation}\label{eqn:cyclic}
\ind_\mu^\Gamma S_{C,V}^\mu = \bigoplus_{(L,W)} n_{L,W}^{C,V} P_{L,W}^\Gamma
\end{equation}
where the sum is over all pairs $(L,W)$ such that $L$ is a quotient of $C$ and $W$ is a simple $k\out(L)$-module. Moreover, by
Proposition \ref{codefres}, to evaluate the above coefficients, we do not need the evaluations of the involved functors at groups
larger than $C$. Thus we can only consider the evaluations at subquotients of $C$.

On the other hand, for any pair $(H,V)$, we have the following isomorphism
\[
\res^{\Gamma(G)}_{\Gamma(H)}\ind_{\mu(G)}^{\Gamma(G)} S_{H,V}^{\mu(G)} \cong
\ind_{\mu(H)}^{\Gamma(H)}\res^{\mu(G)}_{\mu(H)} S_{H,V}^{\mu(G)}
\]
of biset functors for $H$. Indeed, by Theorem \ref{isomtheo}, we have
\[
\ind_{\mu(G)}^{\Gamma(G)} S_{H,V}^{\mu(G)} \cong \ind_{\rho(G)}^{\Gamma(G)} S_{H,V}^{\rho(G)}.
\]
Now denote the identity element of $\Gamma(H)$ by $1_H$ and use the definition of the restriction functor to write
\[
\res^{\Gamma(G)}_{\Gamma(H)}\ind_{\mu(G)}^{\Gamma(G)} S_{H,V}^{\mu(G)} \cong
\res^{\Gamma(G)}_{\Gamma(H)}\ind_{\rho(G)}^{\Gamma(G)} S_{H,V}^{\rho(G)} =1_H\ind_{\rho(G)}^{\Gamma(G)} S_{H,V}^{\rho(G)}
= 1_H\Gamma(G)\otimes_{\rho(G)} S_{H,V}^{\rho(G)}.
\]
Then since $1_H$ acts on the simple functor $S_{H,V}^{\rho(G)}$ trivially, we further have
\begin{eqnarray*}
1_H\Gamma(G)\otimes_{\rho(G)} S_{H,V}^{\rho(G)} &=& 1_H\Gamma(G)\otimes_{\rho(G)} 1_H S_{H,V}^{\rho(G)} \\
&=& 1_H\Gamma(G)1_H \otimes_{\rho(G)} 1_HS_{H,V}^{\rho(G)}\\
&=& \Gamma(H)\otimes_{\rho(G)} 1_HS_{H,V}^{\rho(G)}
\end{eqnarray*}
where the equality $\Gamma(H) = 1_H\Gamma(G) 1_H$ holds by the definition of the alchemic algebra $\Gamma(H)$. In the above
tensor product, we can change $\rho(G)$ by $\rho(H)$ since the elements of the algebra $\rho(G)$ which are not contained in 
$\rho(H)$ are redundant in the tensor product. Thus we get
\begin{eqnarray*}
\res^{\Gamma(G)}_{\Gamma(H)}\ind_{\mu(G)}^{\Gamma(G)} S_{H,V}^\mu &\cong& \Gamma(H) \otimes_{\rho(H)} 1_H 
S_{H,V}^{\rho(H)}\\
&=& \ind_{\rho(H)}^{\Gamma(H)}\res^{\rho(G)}_{\rho(H)} S_{H,V}^{\rho(G)}. 
\end{eqnarray*}
Moreover since $S_{H,V}^{\rho(H)} = \res^{\rho(G)}_{\rho(H)} S_{H,V}^{\rho(G)}$ holds by the definition of this simple functor, we get
\begin{eqnarray*}
\res^{\Gamma(G)}_{\Gamma(H)}\ind_{\mu(G)}^{\Gamma(G)} S_{H,V}^\mu &\cong& \ind_{\rho(H)}^{\Gamma(H)}S_{H,V}^{\rho(H)}\\
&\cong& \ind_{\mu(H)}^{\Gamma(H)} S_{H,V}^{\mu(H)}.
\end{eqnarray*}
Finally, by \cite[Proposition 4.2.2]{B96}, we have $S_{H,V}^{\mu(H)} = \res^{\mu(G)}_{\mu(H)} S_{H,V}^{\mu(G)}$, which completes 
the proof of the following lemma.
\begin{lem}
Let $G$ be a finite group and $H$ be a subquotient of $G$. Then for any simple $k\out(H)$-module $V$, there is an isomorphism
\[
\res^{\Gamma(G)}_{\Gamma(H)}\ind_{\mu(G)}^{\Gamma(G)} S_{H,V}^{\mu(G)} \cong
\ind_{\mu(H)}^{\Gamma(H)}\res^{\mu(G)}_{\mu(H)} S_{H,V}^{\mu(G)}
\]
of biset functors for $H$.
\end{lem}

In particular, if we apply the restriction functor $\res^{\Gamma(G)}_{\Gamma(C)}$ to Equation \ref{eqn:cyclic} and use the above
lemma, we get
\begin{equation}\label{eqn:7}
 \bigoplus_{(L,W)} n_{L,W}^{C,V} \res^{\Gamma(G)}_{\Gamma(C)}P_{L,W}^\Gamma = 
 \res^{\Gamma(G)}_{\Gamma(C)}\ind_\mu^\Gamma S_{C,V}^\mu =
  \ind^{\Gamma(C)}_{\mu(C)}\res_{\mu(C)}^{\mu(G)} S_{C,V}^\mu 
\end{equation}

On the other hand, by \cite[Proposition 3.3]{W}, we have
\[
P_{L,W}^{\Gamma(C)} \cong \res^{\Gamma(G)}_{\Gamma(C)} P_{L,W}^\Gamma.
\]
Furthermore, since, by \cite{Ba} and \cite{B96}, the alchemic algebra $\Gamma(C)$ is semisimple, we have 
$P_{L,W}^{\Gamma(C)} = S_{L,W}^{\Gamma(C)}$ for any pair $(L,W)$. Finally, by \cite[Proposition 4.2.2]{B96}, we have 
$\res^{\mu(G)}_{\mu(C)}S_{C,V}^{\mu} = S_{C,V}^{\mu(C)}$. 
Thus Equation \ref{eqn:7} becomes
\begin{equation}
  \ind^{\Gamma(C)}_{\mu(C)}S_{C,V}^{\mu(C)} 
= \bigoplus_{(L,W)} n_{L,W}^{C,V} S_{L,W}^{\Gamma(C)} 
\end{equation}
where the coefficients $n_{L,W}^{C,V}$ are the ones determined by Equation \ref{eqn:cyclic}. However, these coefficients are 
also given by
\[
n_{L,W}^{C,V} = \dim_k\hom_{\Gamma(C)}(\ind_{\mu(C)}^{\Gamma(C)} S_{C,V}^{\mu(C)}, S_{L,W}^{\Gamma(C)}).
\]
To evaluate the right hand side, note that by \cite[Theorem 6.7]{C3}, there is an isomorphism
\[
\ind_{\nabla(C)}^{\Gamma(C)} S_{L,W}^{\nabla(C)} \cong \coind_{\Delta(C)}^{\Gamma(C)} S_{L,W}^{\Delta(C)}
\]
of biset functors, which implies, by \cite[Proposition 6.2]{C3} that there is an isomorphism 
\[
S_{L,W}^{\Gamma(C)} \cong  \coind_{\Delta(C)}^{\Gamma(C)} S_{L,W}^{\Delta(C)}
\]
of biset functors. Therefore, we obtain
\[
n_{L,W}^{C,V} = \dim_k\hom_{\Gamma(C)}(\ind_{\mu(C)}^{\Gamma(C)} S_{C,V}^{\mu(C)}, \coind_{\Delta(C)}^{\Gamma(C)}
 S_{L,W}^{\Delta(C)}).
\]
Now using adjointness properties of the restriction and induction-coinduction functors together with the equivalences given in 
Theorem \ref{thm:alchemic-equivs}, we can evaluate the right hand side of the above equality as follows. 
\begin{eqnarray*}
\hom_{\Gamma(C)}(\ind_{\mu(C)}^{\Gamma(C)} S_{C,V}^{\mu(C)}, \coind_{\Delta(C)}^{\Gamma(C)} S_{L,W}^{\Delta(C)}) &\cong& 
\hom_{\Gamma(C)}(\ind_{\rho(C)}^{\Gamma(C)} S_{C,V}^{\rho(C)}, \coind_{\Delta(C)}^{\Gamma(C)} S_{L,W}^{\Delta(C)}) \\ &\cong&
\hom_{\nabla(C)}(\ind_{\rho(C)}^{\nabla(C)} S_{C,V}^{\rho(C)}, \res^{\Gamma(C)}_{\nabla(C)}\coind_{\Delta(C)}^{\Gamma(C)} S_{L,W}^{\Delta(C)})\\ &\cong&
\hom_{\nabla(C)}(\ind_{\rho(C)}^{\nabla(C)} S_{C,V}^{\rho(C)}, \coind_{\Omega(C)}^{\nabla(C)}\res^{\Delta(C)}_{\Omega(C)} S_{L,W}^{\Delta(C)}) \\ &\cong&
\hom_{\Omega(C)}(\res_{\Omega(C)}^{\nabla(C)}\ind_{\rho(C)}^{\nabla(C)} S_{C,V}^{\rho(C)}, \res^{\Delta(C)}_{\Omega(C)} S_{L,W}^{\Delta(C)})\\ &\cong&
\hom_{\Omega(C)}(\res_{\Omega(C)}^{\nabla(C)}\ind_{\rho(C)}^{\nabla(C)} S_{C,V}^{\rho(C)}, S_{L,W}^{\Omega(C)})\\ &\cong&
\hom_{k\out(L)}(\ind_{\rho(C)}^{\nabla(C)} S_{C,V}^{\rho(C)} (L), W).
\end{eqnarray*}
With the above isomorphisms, we have proved the following theorem.
\begin{thm}
Let $G$ be a finite group and $C$ be a cyclic subquotient of $G$. Let $V$ be a simple $k\out(C)$-module. Then there is an 
isomorphism 
\[
\ind_\mu^\Gamma S_{C,V}^\mu \cong \bigoplus_{L,W} n_{L,W}^{C,V} P_{L,W}^\Gamma.
\]
of biset functors where $n_{L,W}^{C,V}$ is the multiplicity of $W$ in the $k\out(L)$-module $\ind_{\rho}^\nabla S_{C,V}^\rho (L)$.
\end{thm}

\section{Induced functors for $p$-groups}
Let $p$ be a prime number and $P$ be a non-cyclic $p$-group. In this section, we prove that the biset functor $\ind_\mu^\Gamma
S_{P,1}^\mu$ is never indecomposable. By Corollary \ref{corollary:top}, we have
\[
\ind_\mu^\Gamma S_{P,1}^\mu \cong P_{P,1}^\Gamma \oplus F
\]
where $F$ is a projective biset functor. Our first aim is to show that if $P$ has order at least $p^3$, then $P_{E_2,1}^\Gamma$ is 
a summand of $F$, where $E_2$ is an elementary abelian group of rank $2$. 

By Proposition \ref{codefres}, we need to prove that the trivial $k\out(P)$-module $k$ is a direct summand of the $k\out(P)$-module
$$\codef^\rho_\Omega\res^\Gamma_\rho S_{E_2,1}^\Gamma (P) = \bigcap_{L< P} \ker (\res^P_L: S_{E_2,1}^\Gamma(P)\to 
S_{E_2,1}\Gamma(L)).$$ 

\noindent In other words, we need to show that there is an element $\zeta\in S_{E_2,1}^\Gamma(P)$ such that
\begin{enumerate}
\item $\zeta$ is $\out(P)$-invariant, that is, for any $\phi\in\out(P)$, we have $\iso_{P,P}^\phi(\zeta) =\zeta$, and
\item for any proper subgroup $L$ of $P$, the restriction $\res^P_L \zeta$ is equal to zero.
\end{enumerate}
Now by \cite[Theorem 10.1]{BT}, we have an isomorphism
\[
S_{E_2,1}^\Gamma \cong kD
\]
of biset functors, where $kD$ is the biset functor of the torsion-free Dade group. The element $\zeta$ with the above 
stated properties exists in the Dade group. Indeed, let
\[
\zeta = \Omega_{M(P)}
\]
be the relative syzygy with respect to the $P$-set $M(P)$, where $M(P)$ is the disjoint union of the sets $P/Q$ as $Q$ runs over
all maximal subgroups of $P$, as defined by Bouc in \cite[Notation 6.2.1]{B}. Then
by \cite[Remark 6.2.2]{B}, the restriction of $\zeta$ to any proper subgroup of $P$ is zero. Moreover by \cite[Corollary 4.1.2]{B}, for any
$\phi\in\out(P)$, we have
\[
\iso_{P,P}^\phi \zeta = \iso_{P,P}^\phi (\Omega_{M(P)}) = \Omega_{{}^\phi M(P)} =\zeta
\]
since any automorphism of $P$ only permutes the maximal subgroups of $P$. In particular, $\zeta$ is $\out(P)$-invariant. Therefore the
$k\out(P)$-module $\codef^\rho_\Omega\res^\Gamma_\rho kD(P)$ contains the trivial $k\out(P)$-module $k$ as a direct summand. Hence
we have proved the following theorem.
\begin{thm}\label{thm:p}
Let $p$ be a prime number and $P$ be a non-cyclic $p$-group of order at least $p^3$. Then 
\[
\ind_\mu^\Gamma S_{P,1}^\mu \cong P_{P,1}^\Gamma \oplus P_{E_2,1}^\Gamma \oplus Z
\]
where $E_2$ is an elementary abelian group of rank $2$ and $Z$ is a (possibly zero) projective biset functor. In particular, the biset
functor $\ind_\mu^\Gamma S_{P,1}^\mu$ is not indecomposable. 
\end{thm}

Next we consider the remaining case of $p$-groups, namely the case where $P=E_2$ is an elementary abelian $p$-group of rank 
$2$. In this case, we claim that the following decomposition holds.
\[
\ind_\mu^\Gamma S_{E_2,1} \cong P_{E_2,1}^\Gamma \oplus P_{C_p,1}^\Gamma \oplus P
\]
where $P$ is a (possibly-zero) projective biset functor whose only indecomposable summands are of the form 
$P_{C_p,W}^\Gamma$ for some non-trivial simple $k\out(C_p)$-module $W$.
To obtain this decomposition, note that by \cite[Theorem 10.1]{BT}, we have $kD(E_2) \cong k$. Moreover since $E_2$ is 
minimal for the functor $kD$, we have $kD(H) = 0$ if the order of $H$ is less than $p^2$. Therefore, we have
\[
kD(E_2) = \codef^\rho_\Omega\res^\Gamma_\rho kD (E_2)\cong k.
\]
Thus $P_{E_2,1}^\Gamma$ appears in $\ind_\mu^\Gamma S_{E_2,1}^\mu$ only once. We also know, by Lemma \ref{lem:burn},
 that $P_{1,1}^\Gamma$ does not appear. Thus, by Corollary \ref{quotient}, the only other possibility is $P_{C_p,W}^\Gamma$, 
 where $C_p$ is a cyclic group of 
order $p$ and $W$ is a simple $k\out(C_p)$-module. We claim that the multiplicity $n_p$ of $P_{C_p,1}^\Gamma$ is equal to 1. By 
the above considerations, the number
$n_p$ is equal to the dimension of the subspace of $\out(E_2)$-invariant elements in the $k\out(E_2)$-module 
$\codef^\rho_\Omega\res^\Gamma_\rho S_{C_p,1}^\Gamma (E_2)$. In other words, we need to show, as in the previous case, that
there is an element $\zeta\in S_{C_p,1}^\Gamma(E_2)$ such that
\begin{enumerate}
\item $\zeta$ is $\out(E_2)$-invariant, that is, for any $\phi\in\out(E_2)$, we have $\iso_{E_2,E_2}^\phi(\zeta) =\zeta$, and
\item for any proper subgroup $L$ of $E_2$, the restriction $\res^{E_2}_L \zeta$ is equal to zero.
\end{enumerate}

In order to show that such an element exists, we first describe the evaluation
of $S_{C_p,1}^\Gamma$ at $E_2$. By \cite[Theorem 4.4]{C3}, we have
\[
S_{C_p,1}^\Gamma (E_2) \cong \frac{\ind_\nabla^\Gamma S_{C_p,1}^\nabla (E_2)}{\mathcal K_{C_p,1}(E_2)}
\]
where $\mathcal K_{C_p,1}(E_2)$ is the intersection of kernels of all morphisms from $E_2$ to a group of smaller order. We 
claim that this submodule is zero and hence there is an isomorphism 
\[
S_{C_p,1}^\Gamma (E_2) \cong \ind_\nabla^\Gamma S_{C_p,1}^\nabla(E_2)
\]
of $k\out(E_2)$-modules. Now by \cite[Section 5]{C3}, we have
\[
\ind_\nabla^\Gamma S_{C_p,1}^\nabla (E_2) = \Big( \bigoplus_{\substack{H\in \mbox{\tiny $\sq$}(E_2),\\ H\cong C_p}} k\Big)_{E_2}
\]
where for any $kG$-module $M$, we write $M_G$ for the largest quotient of $M$ on which $G$ acts trivially. Here note that, although
the sum in \cite{C3} should be over all subquotients of $E_2$, in our case, we can restrict to the subquotients of order $p$ since the 
functor $S_{C_p,1}^\nabla$ is non-zero only on the isomorphism class of the group $C_p$.

Moreover, by Section 5
of \cite{C3}, the group $\out(E_2)$ acts on the sum by permuting the components. Since the $E_2$-action on its subquotients
is trivial, we have
\[
\ind_\nabla^\Gamma S_{C_p,1}^\nabla (E_2) = \bigoplus_{\substack{H\in \mbox{\tiny $\sq$}(E_2)\\ H\cong C_p}} k.
\]
Now the group $E_2$ has $p+1$ subgroups of order $p$, say $A_0, A_1, \ldots, A_p$ and $p+1$ corresponding
quotients $Q_0 := E_2/A_0, \ldots, Q_p:= E_2/A_p$, each of order $p$. Thus we have
\[
\ind_\nabla^\Gamma S_{C_p,1}^\nabla (E_2) \cong \bigoplus_{i=0}^p k \oplus \bigoplus_{j=0}^p k
\]
as permutation $k\out(E_2)$-modules. Furthermore, by the proof of Theorem 5.1 in \cite{C3}, the module $\ind_\nabla^\Gamma 
S_{C_p,1}^\nabla (E_2)$ is generated by the elements of the form $\tin_H^{E_2}\otimes a$ where $H$ is a subquotient of $E_2$ 
and $a\in S_{C_p,1}^\nabla (H)$. Since $S_{C_p,1}^\nabla (H)$ is one-dimensional when $H$ is isomorphic to $C_p$, we can 
rewrite the above equality in the following form
\[
\ind_\nabla^\Gamma S_{C_p,1}^\nabla (E_2) \cong \bigoplus_{i=0}^p k\tin_{A_i}^{E_2}\otimes 1 \oplus \bigoplus_{j=0}^p 
k\tin_{Q_j}^{E_2}\otimes 1
\]
still as permutation $k\out(E_2)$-modules. Now to determine the submodule $\mathcal K_{C_p,1}(E_2)$, we need to evaluate the restrictions and
deflations of the above basis elements. By the descriptions of these maps given in Section 5 of \cite{C3}, we have
\[
\res^{E_2}_{A_i} \tin_{A_j}^{E_2}\otimes 1 = |A_i\backslash E_2/A_j| \tin_{A_i\cap A_j}^{A_i}\res^{A_j}_{A_i\cap A_j}\otimes 1
=  \left\{
        \begin{array}{ll}
            p(\iso_{A_i}\otimes 1) & \quad \mbox{\rm{if}}\, i=j \\
            0 & \quad \mbox{\rm{otherwise}}.
        \end{array}
    \right.
\]
Here to get the second row of the above equality, note that, when $i\neq j$, we have $A_i\cap A_j =1$ and by the definition of the 
simple functor $S_{C_p,1}^\nabla$, the restriction of any element to a group of order smaller than $p$ is zero. In the first 
row, we put $\iso_{A_i}:= \iso_{A_i,A_i}^{\mbox{\rm\tiny id}}$. Similarly, we have the following equalities.
\[
\res^{E_2}_{A_i} \tin_{E_2/A_j}^{E_2}\otimes 1
=  \left\{
        \begin{array}{ll}
            \iso_{A_i}\otimes 1 & \quad \mbox{\rm{if}}\, i\neq j \\
            0 & \quad \mbox{\rm{otherwise}}.
        \end{array}
    \right.
\]
\[
\defl^{E_2}_{E_2/A_i} \tin_{A_j}^{E_2}\otimes 1
=  \left\{
        \begin{array}{ll}
            \iso_{A_i}\otimes 1 & \quad \mbox{\rm{if}}\, i\neq j \\
            0 & \quad \mbox{\rm{otherwise}}.
        \end{array}
    \right.
\]
\[
\defl^{E_2}_{E_2/A_i} \tin_{E_2/A_j}^{E_2}\otimes 1
=  \left\{
        \begin{array}{ll}
            \iso_{A_i}\otimes 1 & \quad \mbox{\rm{if}}\, i= j \\
            0 & \quad \mbox{\rm{otherwise}}.
        \end{array}
    \right.
\]
Now let 
\[
x= \sum_i a_i\tin_{A_i}^{E_2}\otimes 1 + \sum_j b_j\tin_{E_2/A_j}^{E_2}\otimes 1
\]
be an element of $\ind_\nabla^\Gamma S_{C_p,1}^\nabla(E_2)$. Then by the above calculations, we have
\[
\res^{E_2}_{A_k}(x) = p\cdot a_k + \sum_{j\neq k} b_j
\]
and 
\[
\defl^{E_2}_{E_2/A_k}(x) = \sum_{i\neq k} a_i +  b_k.
\]
Therefore $x\in\mathcal K_{C_p,1}(E_2)$ if and only if the equations
\[
p\cdot a_k + \sum_{j\neq k} b_j =0\quad  \mbox{\rm{and}}\quad \sum_{i\neq k} a_i +  b_k =0
\]
are satisfied for all $k$. Note that if we order the subquotients of order $p$ as $A_0, A_1,\ldots, A_p,$  $Q_0,\ldots, Q_p$, the 
corresponding coefficient matrix is of the form
\[
C=\left[
\begin{array}{cc}
p\cdot I & A\\ 
A & I
\end{array}\right]
\]
where $I$ is the identity matrix of size $p+1$ and $A$ is the matrix 
$$
\left[\begin{array}{cccc}
0&1&\cdots &1\\
1&0&\cdots &1\\
\vdots & &\ddots &\vdots \\
1&1&\cdots &0\\
\end{array}\right]
$$
which has zeros on the diagonal and 1 at any other entry. Straightforward calculations show that the $(2p+2)\times (2p+2)$-matrix 
$C$ is non-singular
(actually it has determinant $(-1)^pp(1-p)^{p+1}$). Therefore the equation $Cx= 0$ has no non-zero solution, hence 
$\mathcal K_{C_p,1}(E_2) = 0$, as required, and we have
\begin{eqnarray}\label{eqn:ate2}
S_{C_p,1}^\Gamma (E_2) \cong \ind_\nabla^\Gamma S_{C_p,1}^\nabla(E_2) = \bigoplus_{H\in\mbox{\tiny $\sq$}(E_2), H\cong C_p} k
\end{eqnarray}
as permutation $k\out(E_2)$-modules. This completes the description of the evaluation of the simple functor $S_{C_p,1}^\Gamma$
at $E_2$. Next we proceed to show that an element $\zeta$ with the above stated properties exists in the evaluation in 
(\ref{eqn:ate2}).

We first determine the intersection of kernels of restrictions to subgroups $A_i$. Let $x$ be as above. Then, by the above 
calculations, we have 
\[
\res^{E_2}_{A_k}(x) = 0 \quad \mbox{\rm for all $k$ if and only if} \quad p\cdot a_k + \sum_{j\neq k}b_j = 0\quad \mbox{\rm for all 
k.}
\]
Thus if the restriction to all subgroups $A_i$ of $x$ is equal to zero, then the coefficients $a_i$ are uniquely determined by the
rest of the coefficients $b_j$ by the formula
\[
a_k = \frac{1}{p} \sum_{j\neq k}b_j.
\]
In particular, we are free to choose the coefficients $b_j$ and hence
\[
\dim_k\codef^\rho_\Omega\res^\Gamma_\rho S_{C_p,1}^\Gamma(E_2) = p+1.
\]
Next we look for $\out(E_2)$-invariant elements in this kernel since the element $\zeta$ whose existence in claimed above is 
$\out(E_2)$-invariant. Now it is clear that $x$ is $\out(E_2)$-invariant if and only if 
$b_i = b_j$ for all $i$ and $j$ since $\out(E_2)$ permutes these coefficients. Thus in the $k\out(E_2)$-module 
$\codef^\rho_\Omega\res^\Gamma_\rho S_{C_p,1}^\Gamma(E_2)$, there is a unique $\out(E_2)$-invariant element, up to a 
constant multiple, and hence we have completed the proof of the following result.
\begin{pro}
Let $p$ be a prime number. Then there is an isomorphism 
\[
\ind_\mu^\Gamma S_{E_2,1}^\mu \cong P_{E_2,1}^\Gamma \oplus P_{C_p,1}^\Gamma \oplus P
\]
of biset functors where $P$ is a (possible zero) projective biset functor whose only indecomposable summands are of the form
$P_{C_p,W}^\Gamma$ for some simple $k\out(C_p)$-module $W$.
\end{pro}
With this proposition, Lemma \ref{lem:burn} and Theorem \ref{thm:p}, the following corollary is immediate.
\begin{cor}\label{cor:non}
Let $p$ be a prime number and $P$ be a $p$-group. Then the biset functor $\ind_\mu^\Gamma S_{P,1}^\mu$ is not
indecomposable.
\end{cor}
\section{Induced functors for simple groups}
In this section, we describe the induced simple native Mackey functors that are indexed by simple groups. We start with the trivial
case of the trivial group. In this case, by the classification of the simple native Mackey functors, there is a unique simple functor 
which is $S_{1,1}^\mu$. Note that, by the isomorphism in (\ref{eqn:simples}), this functor is the constant functor $\underline k$ which
takes the value $k$ at any subquotient of $G$ and maps any morphism to the identity homomorphism $k\to k$.

By Corollary \ref{quotient}, we immediately get the isomorphism   
\[
\ind_\mu^\Gamma S_{1,1}^\mu \cong P_{1,1}^\Gamma
\]
of biset functors. Moreover by \cite[Remark 5.1.3]{B96}, the projective indecomposable biset functor $P_{1,1}^\Gamma$ is 
isomorphic to the Burnside functor $kB^G$. Thus we obtain the following result.
\begin{pro} There is an isomorphism
\[
\ind_\mu^\Gamma S_{1,1}^\mu \cong kB^G
\]
 of biset functors.
\end{pro}

Note that, in this example, the trivial group has no non-trivial proper quotients. This property is also shared by all simple groups. Our next aim
is to determine the induction of the simple native Mackey functors parameterized by simple groups, that is, the biset functor $\ind_\mu^\Gamma
S_{H,V}^\mu$ where $H$ is a simple group. By Corollary \ref{quotient}, this functor has at most two summands, namely $P_{H,V}^\Gamma$
and $P_{1,1}^\Gamma$ with the first one of multiplicity one. Thus we only need to determine the multiplicity of $kB$ in
the induced functor.

Next we determine the multiplicity $m_B$ of the Burnside functor $kB^G = P_{1,1}^\Gamma$ in the induced 
functor $\ind_\mu^\Gamma S_{H,V}^\mu$ for any pair $(H,V)$ and then specialize to the case of simple groups. First, by 
Proposition \ref{codefres}, we have
\[
m_B = \dim_k\hom_{k\out(H)}(V,\codef^\rho_\Omega\res^\Gamma_\rho S_{1,1}^\Gamma(H)).
\]
Also, by \cite[Proposition 4.4.8]{B96}, we have
\[
S_{1,1}^\Gamma \cong k\mathcal R_{\mathbb Q}
\]
where $\mathcal R_{\mathbb Q}$ is the functor of rational characters of finite groups. As in the previous sections, we need to 
determine the $\out(H)$-invariant elements in $\codef^\rho_\Omega\res^\Gamma_\rho S_{1,1}^\Gamma(H)$. For this aim, we first 
determine the restriction of $S_{1,1}^\Gamma$ to the category of native Mackey functors. By the above identification, this amounts 
to determine the simple summands of the native Mackey functor $k\mathcal R_{\mathbb Q}$ of rational representations. Note further 
that, by the isomorphism given in (\ref{eqn:twin}), we have
\[
k\mathcal R_{\mathbb Q} \cong \ind_{\rho}^{\mu}\inf_{\Omega}^{\rho}\defl_{\Omega}^{\tau}\res^{\mu}_{\tau} k\mathcal R_{\mathbb Q}.
\]
On the other hand, by the definition of the deflation functor $\defl^\tau_\Omega$, for any $H\in \sq(G)$, we have
\[
\defl^\tau_\Omega k\mathcal R_{\mathbb Q} (H) = k\mathcal R_{\mathbb Q} (H) \Big/ \sum_{L < H} {\mbox{\rm Im}} (\ind_L^H).
\] 
Now, by Artin's Induction Theorem, the above quotient is non-zero only if the group $H$ is cyclic. Moreover, when the group $H$
is cyclic, the quotient is isomorphic to $k$ as $k\out(H)$-modules. Indeed, in this case, by \cite[Theorem 30]{S}, the set of characters 
$\ind_L^H 1$ as $L$ runs over all subgroups of $H$ generates $k\mathcal R_{\mathbb Q}(H)$ and moreover, as indicated in Exercise 
13.8 in \cite{S}, this set is actually a basis of $k\mathcal R_{\mathbb Q}(H)$. It is now clear that the quotient is one dimensional. 
Moreover the $\out(H)$-action on this vector space is trivial since $k\out(H)$ fixes the character $1=\ind_H^H 1$.

With this observation, we conclude that there is an isomorphism 
\[
\defl^\tau_\Omega\res^\Gamma_\tau S_{1,1}^\Gamma \cong \bigoplus_{\substack{H\in\tiny{\sq}(G): H: \mbox{\rm{\scriptsize{cyclic}}}
\\ \mbox{\rm \scriptsize{up to isomorphism}}}} S_{H,1}^\Omega
\]
of $\Omega$-modules. Thus by the above isomorphism and the isomorphism given in Equation (\ref{eqn:simples}), we get that 
following result.
\begin{pro}\label{thm:nativerat}
As a native Mackey functor for $G$ over $k$, the functor of rational representations decomposes as
\[
k\mathcal R_{\mathbb Q} \cong  \bigoplus_{\substack{H\in\tiny{\sq}(G): H: \mbox{\rm{\scriptsize{cyclic}}}
\\ \mbox{\rm \scriptsize{up to isomorphism}}}} S_{H,1}^\mu.
\]
\end{pro}

Recall that we are looking for elements $\zeta$ in $k\mathcal R_{\mathbb Q}(H)$ such that
\begin{enumerate}
\item $\zeta$ is $\out(H)$-invariant, that is, for any $\phi\in\out(H)$, we have $\iso_{H,H}^\phi(\zeta) =\zeta$, and
\item for any proper subgroup $L$ of $H$, the restriction $\res^{H}_L \zeta$ is equal to zero.
\end{enumerate}

The elements satisfying the second property lies in 
\[
\codef^\rho_\Omega\res^\Gamma_\rho k\mathcal R_{\mathbb Q}(H) = \bigcap_{K< H} \ker(\res^H_K:k\mathcal R_{\mathbb Q}(H)\to
k\mathcal R_{\mathbb Q}(K)).
\]

Now by Equation (\ref{eqn:simples}) and by Proposition \ref{thm:nativerat}, we have an isomorphism
\[
k\mathcal R_{\mathbb Q} \cong \coind_\tau^\mu\infl_\Omega^\tau\defl^\tau_\Omega\res^\mu_\tau k\mathcal R_{\mathbb Q}
\]
of native Mackey functors. Therefore, we have
\begin{eqnarray*}
\codef^\rho_\Omega\res^\mu_\rho k\mathcal R_{\mathbb Q} &\cong & \codef^\rho_\Omega\res^\mu_\rho\coind_\tau^\mu\infl_\Omega^\tau\defl^\tau_\Omega\res^\mu_\tau k\mathcal R_{\mathbb Q}\\
&\cong& \codef^\rho_\Omega\coind^\rho_\Omega\res_\Omega^\tau\infl_\Omega^\tau\defl^\tau_\Omega\res^\mu_\tau k\mathcal R_{\mathbb Q}
\end{eqnarray*}
where we use Theorem \ref{thm:native-equivs} to interchange the coinduction and the restriction functors. Then, again by Theorem
\ref{thm:native-equivs}, the composition $\codef^\rho_\Omega\coind^\rho_\Omega$ and the composition 
$\res_\Omega^\tau\infl_\Omega^\tau$ are naturally equivalent to the identity functor. Therefore, we get that
\[
\codef^\rho_\Omega\res^\mu_\rho k\mathcal R_{\mathbb Q} \cong \defl^\tau_\Omega\res^\mu_\tau k\mathcal R_{\mathbb Q}.
\]
In particular,  
\[
\codef^\rho_\Omega\res^\mu_\rho k\mathcal R_{\mathbb Q}(H) = 0
\]
unless $H$ is cyclic and hence no element $\zeta$ with the above prescribed properties exists in $k\mathcal R_{\mathbb Q}
(H)$. Therefore when $H$ is not cyclic, the Burnside functor is not a direct summand of the biset functor 
$\ind_\mu^\Gamma S_{H,V}$.

On the other hand, if $H$ is cyclic, then by the above results,
\[
\codef^\rho_\Omega\res^\mu_\rho k\mathcal R_{\mathbb Q}(H) \cong k
\]
as $k\out(H)$-modules. Therefore all elements in this module are $\out(H)$-fixed and this completes the proof of the following 
lemma.
\begin{lem}\label{lem:burn} The multiplicity $m_B$ of the Burnside functor $kB^G$ as a direct summand in the biset functor 
$\ind_\mu^\Gamma S_{H,V}^\mu$ is given by 
\[
m_B = \left\{
        \begin{array}{ll}
            1 & \quad \mbox{\it if $H$ is cyclic and $V$ is trivial}. \\
            0 & \quad otherwise.
        \end{array}
    \right.
\]
\end{lem}
Combining the above lemma with Corollary \ref{quotient}, we get the following theorem.
\begin{thm} Let $H$ be a simple group. 
\begin{enumerate}
\item If $H$ is of prime order, and $1$ is the trivial $k\out(H)$-module, then 
\[
\ind_\mu^\Gamma S_{H,1}^\mu \cong P_{H,1}^\Gamma \oplus P_{1,1}^\Gamma.
\]
\item Otherwise, if $H$ is cyclic but $V$ is non-trivial, or if $H$ is not cyclic, then 
\[
\ind_\mu^\Gamma S_{H,V}^\mu \cong P_{H,V}^\Gamma.
\]
\end{enumerate}
\end{thm}
Note that when $H$ is a simple, non-cyclic group and $V$ is non-trivial, the functor $\ind_\mu^\Gamma S_{H,V}^\mu$ gets 
easier. Indeed, in this case, we have $\ind_\rho^\nabla S_{H,V}^\rho \cong S_{H,V}^\nabla$ by Theorem \ref{pro:rhoTonabla}. 
To prove this isomorphism, note that, for any $K\in\sq(G)$, we have 
\[
\ind_\rho^\nabla S_{H,V}^\rho (K) = 0
 \]
if $K$ is not isomorphic to a quotient of $H$. But since $H$ is simple this means that the above evaluation is equal to zero
unless $K = H$ or $K = 1$. Now the evaluation at $H$ is equal to $V$ by the definition of the induction functor. Also, we have
\[
\ind_\rho^\Delta S_{H,V}'\rho (1) = kX(1,H)\otimes_{k\out(H)} V \cong k\otimes_{k\out(H)} V = 0.
\]
Therefore we get
\[
\ind_\rho^\nabla S_{H,V}^\rho(1) = 0
\]
and hence the isomorphism
\[
\ind_\rho^\nabla S_{H,V}^\rho \cong S_{H,V}^\nabla
\]
of $\nabla$-modules hold.
Note that, when $V=1$ is the trivial $k\out(H)$-module, we have $k\otimes_{k\out(H)} k = k$, and hence 
$\ind_\rho^\nabla S_{H,V} (1) = k$. In particular, the above isomorphism does not hold when $V$ is trivial.
We can summarize this result as follows.
\begin{cor} Assume the notation of the above theorem and suppose that $H$ is not cyclic and $V$ is non-trivial. Then there is an 
isomorphism
$$P_{H,V}^\Gamma \cong \ind_\mu^\Gamma S_{H,V}^\mu \cong \ind_\nabla^\Gamma S_{H,V}^\nabla$$
of biset functors.
\end{cor}
Note that the induction functor $\ind_\nabla^\Gamma$ can be evaluated using \cite[Theorem 5.1]{C3}.

\begin{prob} By Corollary \ref{cor:non} and Lemma \ref{lem:burn}, the biset functor $\ind_\mu^\Gamma S_{H,1}^\mu$ is not
indecomposable if $H$ is a $p$-group for some prime $p$ or a cyclic group. Also in some other groups of small order, not 
covered by these results, we showed, by hand, that the above functor has at least two direct summands. We do not include these
calculations. However we could not answer
whether $\ind_\mu^\Gamma S_{H,V}^\mu$ being indecomposable (and hence isomorphic with $P_{H,V}^\Gamma$)
characterizes simple groups, and leave it as open.
\end{prob}
\appendix
\section{The functor $\ind_\rho^\nabla$}

In this section, we describe the functor $\ind_{\rho(G)}^{\nabla(G)} = \nabla(G)\otimes_{\rho(G)}$. For simplicity, we write $\nabla 
= \nabla(G)$ and $\rho = \rho(G)$. For our aims, it is sufficient to describe the functor
only for atomic functors, that is, for functors which are non-zero only on one isomorphism class of subquotients. Thus, let 
$D_{H,V}$ be the $\rho$-module such that $D_{H,V}(K) = 0$ if $K$ is not isomorphic to $H$ and $D_{H,V}(H) = V$. Note that,
here, the module $V$ is not necessarily simple. Further let $K$ be a subquotient of $G$. We want to describe
\[
M:=\ind_\rho^\nabla D_{H,V} (K) = \iso_{K,K}^{{\rm id}}\nabla\otimes_\rho D_{H,V}
\]
where $\iso_{K,K}^{{\rm id}}$ is the $(K,K)$-biset $K$ where the group $K$ acts on both sides via multiplication.
By \cite[Theorem 3.3]{C3}, the algebra $\nabla$
has a $k$-basis consisting of all triples $[R,L,\phi:L\ra K]$ where $R$ runs over all subquotients of $G$, $L$ runs over all 
subquotients 
of $R$, up to conjugation, $K$ runs over all subquotients of $G$ and $\phi:L\ra K$ is an isomorphism, taken up to $(K,R)$-
conjugacy, in the sense of \cite[Section 3]{C3}. Thus we can write
\begin{eqnarray*}
M &=&\iso_{K,K}^{{\rm id}} \nabla\otimes_\rho e_HD_{H,V}\\
&=& \bigoplus_{[R,L,\phi:L\ra K], R\cong H} k\iso_{K,L}^\phi\des^R_Le_H\otimes_\rho e_HD_{H,V}.
\end{eqnarray*}
where $e_H$ is the sum of $\iso_{R,R}^{{\rm id}}$ where $R\cong H$ runs over all subquotients of $G$ isomorphic to $H$. Now 
if $L$ is a proper subgroup of $R$, then the corresponding summand will be equal to zero since the restriction map $\des^R_L$
annihilates $D_{H,V}$. Thus the sum becomes
\begin{eqnarray*}
M &=& \bigoplus_{\substack{[R,R/N,\phi:R/N\ra K], \\R\cong H}} k\iso_{K,R/N}^\phi\des^R_{R/N}e_H\otimes_\rho e_HD_{H,V}.
\end{eqnarray*}
Let $\iso(L,K)$ denote the set of all isomorphisms from $L$ to $K$. Then the above sum can be written as follows.
\begin{eqnarray*}
M &=& \bigoplus_{\substack{R\in {\tiny\sq}(G), N\unlhd R \\R\cong H, R/N\cong K \\ \phi\in\iso(R/N,K)}} k\iso_{K,R/N}^\phi\des^R_{R/N}e_H\otimes_\rho e_HD_{H,V}.
\end{eqnarray*}
Now there is a left $K$-action and a right $R$ action on the set $\iso(R/N,K)$ which are trivial on the tensor product. Thus the 
above sum reduces to 
\begin{eqnarray*}
M &=& \bigoplus_{\substack{R\in {\tiny\sq}(G), N\unlhd R \\R\cong H, R/N\cong K \\ \phi\in K\backslash\iso(R/N,K)/R}} k\iso_{K,R/N}^\phi\des^R_{R/N}e_H\otimes_\rho e_HD_{H,V}\\
&=& \bigoplus_{\substack{R\in {\tiny\sq}(G), N\unlhd R \\R\cong H, R/N\cong K \\ \phi\in K\backslash\iso(R/N,K)/R}} k\iso_{K,R/N}^\phi\des^R_{R/N}\otimes_{k\out(R)} V(R)\\
\end{eqnarray*}
where $V(R) := D_{H,V}(R)$. To simplify the above sum, let 
\[
X(K,R) = \bigsqcup_{N\unlhd R: R/N\cong K} K\backslash \iso(R/N,K) / R.
\]
The set $X(K,R)$ is an $(\out(K),\out(R))$-biset, via pre and post compositions. Now we have
\[
M= \bigoplus_{R\in{\tiny\sq}(G): R\cong H} kX(K,R) \otimes_{k\out(R)} V(R).
\]
This completes the proof of the following result.
\begin{thm}\label{pro:rhoTonabla} Let $V$ be a $k\out(H)$-module. Denote by $D_{H,V}$ the $\rho$-module 
which is zero on all
subquotients of $G$ not isomorphic to $H$ and $D_{H,V}(H) = V$. Let $K$ be a subquotient of $G$. Then
\[
\ind_\rho^\nabla D_{H,V}(K) = \bigoplus_{R\in{\tiny\sq}(G): R\cong H} kX(K,R) \otimes_{k\out(R)} V(R)
\]
where $X(K,R)$ is the $(\out(K),\out(R))$-biset defined above.
\end{thm}

\end{document}